\documentclass{siamart251216}

\usepackage{lipsum}
\usepackage{amsfonts}
\usepackage{graphicx}
\usepackage{epstopdf}
\usepackage{algorithmic}
\usepackage{comment,enumitem}
\usepackage{subcaption}
\usepackage{diagbox,mathtools}
\usepackage{multirow}
\usepackage{colortbl}
\usepackage{xcolor}
\usepackage{tikz}
\newtheorem{remark}{Remark}

\ifpdf%
 \DeclareGraphicsExtensions{.eps,.pdf,.png,.jpg}
\else
 \DeclareGraphicsExtensions{.eps}
\fi
\usepackage{amsopn}

\usepackage{booktabs}

\newcommand{\bm}[1]{\mathbf{#1}}
\newcommand{\bs}[1]{\boldsymbol{#1}}

\usepackage[textsize=tiny]{todonotes}

\newcommand{\TheShortTitle}{%
 Two-Grid Framework for AA and NGMRES
}

\headers{\TheShortTitle}{J.~Adler,  Y.~He.,  X.~Hu,  S.~Lefebvre}

\title{A Unified Two-Grid Framework for Anderson Acceleration and Nonlinear GMRES\thanks{The work of Adler and Hu is partially supported by the National Science Foundation (NSF) under grant DMS-2513394.}}

\author{James H. Adler\thanks{Department of Mathematics, Tufts University, Medford, 02155 MA, USA (james.adler@tufts.edu,\, xiaozhe.hu@tufts.edu,\, Satchel.Lefebvre@tufts.edu)} \and Yunhui He\thanks{Department of Mathematics, University of Houston, 3551 Cullen Blvd, Room 641, Houston, Texas 77204-3008, USA (yhe43@central.uh.edu)}\and Xiaozhe Hu\footnotemark[2]\and Satchel Lefebvre\footnotemark[2]
}

\ifpdf%
\hypersetup{%
 pdftitle={TITLE},
 pdfauthor={J.~Adler,  Y.~He. X.~Hu,  S.~Lefebvre}
}
\fi

\begin{document}

\maketitle


\begin{abstract}
In this work, we recast two widely used acceleration methods for fixed-point iterations, Anderson acceleration (AA) and the nonlinear generalized minimal residual method (NGMRES), as two-grid methods. By explicitly deriving the error propagation matrices for AA and NGMRES on linear problems, we show that both methods, together with several of their existing variants, can be reformulated as two-grid methods with a single pre- or postsmoothing step and a coarse-grid correction given by a projection with respect to a suitable inner product. This reformulation not only unifies AA, NGMRES, and their variants within a single algorithmic framework, but also enables the design of new variants by independently adjusting the components of the two-grid method, such as the coarse space and the smoothing steps. In particular, we propose a new variant that enlarges the coarse space by incorporating the most recently updated residual, an improvement that is naturally revealed by the two-grid formulation but is hidden in the standard formulations of the AA and NGMRES algorithms. Numerical experiments on both an SPD Poisson problem and a family of non-SPD convection-diffusion problems show that this new variant often outperforms both AA and NGMRES, as well as their variants.
\end{abstract} 



\section{Introduction}\label{sec:intro}
Anderson acceleration (AA) \cite{anderson1965iterative} and the nonlinear generalized minimal residual method (NGMRES) \cite{washio1997krylov} are widely used for accelerating fixed-point iterative methods. Both approaches utilize previous iterates and solve small-dimensional optimization problems to find coefficients in order to compute the next iterate. AA constructs a linear combination of previous iterates that minimizes the norm of an associated residual, thereby accelerating the convergence of the underlying fixed-point iteration. On the other hand, NGMRES was proposed in \cite{washio1997krylov,oosterlee2000krylov} as a method that minimizes the $(k+1)$-th standard residual at each iteration to accelerate nonlinear multigrid methods.  Both AA and NGMRES have since been widely applied to a variety of problems, see \cite{washio1997krylov,oosterlee2000krylov,walker2011anderson,Lott2012flow,pollock2019anderson}.

While AA was introduced in \cite{anderson1965iterative}, theoretical study of its acceleration properties came much later with the works of \cite{Fang2009classesNL_accel, walker2011anderson}. Since then, a multitude of works have examined the convergence analysis of AA and NGMRES, e.g., \cite{toth2015convergence,evans2020proof,de2021asymptotic, He2026NGMRES_converge}. Additionally, \cite{potra2013characterization} provided a complete characterization of the behavior of AA using all the previous iterates.  The common consensus is that both AA and NGMRES bear a strong resemblance and connection to GMRES. For instance, \cite{walker2011anderson} showed that AA is equivalent to GMRES under certain conditions. Similarly, when NGMRES is applied to Richardson's iteration for linear systems, \cite{greif2026convergenceLinGMRES} showed that NGMRES is equivalent to GMRES and, based on these relationships, connected AA and NGMRES together with GMRES.

Besides studying the convergence properties of AA and NGMRES, another active area of research is the development of variants of both methods that could further improve convergence; see, for example, \cite{chen2022composite, Chen2023Initial,he2026gen_alt_AA, he2026altGMRES}. A typical approach to deriving variants is based on how the norm and the least-squares problem associated with AA and NGMRES are defined and solved. The classical AA and NGMRES use the $\ell^2$ norm for the least-squares problem. In \cite{toth2015convergence}, the authors demonstrate that other norms could be used with reformulation. Additionally, different choices of norms could be considered for different applications, see \cite{hawkins2025choice,heRebholzNGMRESNSE,he2026modified}.
Alternatively, \cite{he2026NGMRESandAAwRich} develops variants still under the $\ell^2$ norm, such as AAr, AAf, and NGMRESf, by exploring different formulations of the least-squares problem.

In this paper, inspired by the error propagation matrix derived for NGMRES in \cite{he2026altGMRES} for linear problems, we first derive an analogous error propagation matrix for AA. By noticing the close relationship between these error propagators and their counterparts in classical two-grid and multigrid methods, we are able to reformulate AA and NGMRES as two-grid methods with one postsmoothing or presmoothing step, respectively.  Ultimately, we propose a unified two-grid acceleration framework, which not only recovers most of the existing variants of AA and NGMRES, but also enables us to develop an even larger class of variants by adjusting different components such as smoothing steps and coarse-grid corrections. Moreover, under the two-grid framework, we can include the most updated residual after presmoothing into the coarse space.  This leads to a specific new variant, TGA($m$; 1, 0), in which we keep only one presmoothing step, thus making the computational cost of each iteration comparable to AA, NGMRES, and their variants.  We perform numerical experiments on the Poisson equation and convection-diffusion problems.  Numerical results show that our new variant consistently outperforms both AA,  NGMRES, and their variants when $m \geq 2$, with its advantage growing as the problem becomes increasingly convection-dominated.   

The rest of this paper is organized as follows.  We present the NGMRES and AA algorithms and their variants in \Cref{sec:AAandNGMRES}, along with a review of general two-grid methods.  Then, in \Cref{sec:AAandNGasMG}, we recast the accelerating methods as two-grid schemes.  In \Cref{sec:TGA}, we combine both methods under one general framework, and from this new formulation, we then introduce a new variant.  Numerical experiments on an SPD problem and a non-SPD problem are next presented in \Cref{sec:numRes} to illustrate the efficiency of the new approach.  Finally, conclusions and a discussion are given in \Cref{sec:conclusion}.

\section{Preliminaries}\label{sec:AAandNGMRES}
In this section, we provide a brief introduction to AA, NGMRES, and two-grid methods applied to the following linear system,
\begin{equation}\label{eqn:Ax=b}
	A\bm{x} = \bm{b},
\end{equation}
where $A \in \mathbb{R}^{n\times n}$ and $\bm{x}$, $\bm{b} \in \mathbb{R}^{n}$.  For the sake of simplicity, we assume $A$ is nonsingular.  A fixed-point iteration for solving \eqref{eqn:Ax=b} is given by
\begin{equation}\label{eqn:fixed-point-iteration}
	\bm{x}^{k+1} := g(\bm{x}^k) = \bm{x}^{k} + B(\bm{b} -A \bm{x}^{k}),
\end{equation}
where $B$ is generally considered as a preconditioner of $A$ such that $B \approx A^{-1}$.  Examples of $B$ include classical relaxation methods such as the Richardson method, $B=\omega I$ with $\omega\in \mathbb{R}$, or the Jacobi method, $B=D^{-1}$, where $D = \operatorname{diag}(A)$.  Next, we define the linear system residual as
\begin{equation}\label{eqn:linear_sysyem_residual}
	r(\bm{x}^k) = \bm{r}^{k} \coloneq \bm{b} -A\bm{x}^k,
\end{equation}
and the fixed-point iteration residual as
\begin{equation}\label{eqn:FP-residual}
	f(\bm{x}^k) = \bm{f}^k \coloneq g(\bm{x}^k) - \bm{x}^k = B\bm{r}^k.
\end{equation}
Finally, the error at the $k$th iteration is defined as
\begin{equation}
	\bm{e}^k =\bm{x} -\bm{x}^k.
\end{equation}

AA, NGMRES, and their variants use previous iterates to accelerate the fixed-point iteration \eqref{eqn:fixed-point-iteration}.  When a fixed number of previous iterates, $m$, is used, the corresponding methods are denoted as AA$(m)$ and NGMRES$(m)$, respectively. At the $k$-th step, the current window size is denoted as $m_k = \min(m,k)$.  When $m=k$ for all $k$ (i.e., $m$ is not fixed), the corresponding methods are referred to as full AA and full NGMRES.  Since one of the structural differences between AA and NGMRES lies in the least-squares problem each uses, we discuss this choice when each algorithm is presented in the following subsections.

\subsection{Nonlinear GMRES}
NGMRES accelerates a fixed-point iteration by\newline forming the next iterate as a linear combination of the current fixed-point update and a number of previous iterates, with the combination coefficients chosen to minimize  
a specific minimization problem at each iteration \cite{greif2026convergenceLinGMRES,he2026NGMRESandAAwRich}.   
The NGMRES algorithm is presented below in \cref{alg:NGMRES}.

\begin{algorithm}[H]
	\caption{NGMRES($m$): Nonlinear GMRES with depth $m$} \label{alg:NGMRES}
	\begin{algorithmic}[1]
\STATE Given $\bm{x}^0$ and fixed-point iteration $g$
\FOR{$k =0,1,...$}
\STATE $m_k = \min(m,k)$.
\STATE Update, \begin{equation}\label{eqn:NGMRES_update}
\bm{x}^{k+1} \gets g(\bm{x}^k) - \sum_{j = 1}^{m_k} \beta_j^{k} \big(\bm{x}^{k-j+1}- \bm{x}^{k-j} \big) - \beta_{m_k+1}^k \big( g(\bm{x}^k) - \bm{x}^k \big).
\end{equation}
where $\bs{\beta}^k= [\beta^k_1, \beta^k_2,...,\beta^k_{m_k+1}]^\top$ is found by solving
\begin{align}
 \min_{\bs{\beta}_k} \Bigg\| r(g(\bm{x}^k)) - \sum_{j = 1}^{m_k} \beta_j^{k} (\bm{r}^{k-j+1}- \bm{r}^{k-j} \big) - \beta_{m_k+1}^k \big( r(g(\bm{x}^k)) - \bm{r}^k \big) \Bigg\|_2^2.\tag{LSP-NG} \label{eqn:NGMRES_LSP}
\end{align}
\ENDFOR
\end{algorithmic}
\end{algorithm}

A variant of NGMRES, called NGMRESf, was proposed in \cite{he2026NGMRESandAAwRich}, and differs from NGMRES by solving the following least-squares problem instead of \eqref{eqn:NGMRES_LSP} in \cref{alg:NGMRES}:
\begin{align}
	\min_{\bs{\beta}_k} \Bigg\| f(g(\bm{x}^k)) - \sum_{j = 1}^{m_k} \beta_j^{k} \big(\bm{f}^{k-j+1}- \bm{f}^{k-j} \big) - \beta_{m_k+1}^k \big( f(g(\bm{x}^k)) - \bm{f}^k \big)\Bigg\|_2^2. \tag{LSP-NGf}\label{eqn:NGMRESf_LSP}
\end{align}
All other steps of NGMRESf are exactly the same as those in \cref{alg:NGMRES} and, for the sake of simplicity, we will not present it separately.

\subsection{Anderson Acceleration}
AA accelerates a fixed-point iteration in a similar fashion to NGMRES.  The differences are that the next iterate is a linear combination of the current fixed-point update and several previous updates, and that a different least-squares problem is solved to find the coefficients used in the linear combination.  The classical AA algorithm is presented below in \cref{alg:AA}.
 
\begin{algorithm}[H]
	\caption{Anderson Acceleration with depth $m$ } \label{alg:AA}
	\begin{algorithmic}[1]
\STATE Given $\bm{x}^0$ and fixed-point iteration $g$
\FOR{$k =0,1,...$}
\STATE Set $m_k = \min(m,k)$.
\STATE Update,
\begin{equation}\label{eqn:AA_update}
 \bm{x}^{k+1} \gets g(\bm{x}^k) - \sum_{j = 1}^{m_k} \alpha_j^{k} \big(g(\bm{x}^{k-j+1}) - g(\bm{x}^{k-j}) \big),
\end{equation}
where $\bs{\alpha}^k = [\alpha^k_1, \alpha^k_2,\ldots,\alpha^k_{m_k}]^\top$ is found by
\begin{align}
 \min_{\boldsymbol{\alpha}^{k}} \Bigg\| \bm{f}^k - \sum_{j=1}^{m_k} \alpha_j^{k} \big( \bm{f}^{k-j+1} - \bm{f}^{k-j} \big) \Bigg\|_2^2. \tag{LSP-AA}\label{eqn:AA_LSP}\nonumber
\end{align}
\ENDFOR
\end{algorithmic}
\end{algorithm}
Several variants of AA have also been proposed in \cite{he2026NGMRESandAAwRich} by solving different least-squares problems.  In particular, the AAr and AAf algorithms solve the following least-squares problems, instead of \eqref{eqn:AA_LSP} in \cref{alg:AA}:
\begin{align}
	&\min_{\boldsymbol{\alpha}^{k}}\Bigg\|r(g(\bm{x}^k))-\sum_{j=1}^{m_k}\alpha_j\Big(r\big(g(\bm{x}^{k-j+1})\big)- r\big(g(\bm{x}^{k-j})\big)\Big) \Bigg\|_2^2,
	\tag{LSP-AAr}\label{eqn:AAr_LSP}\\
	&\min_{\boldsymbol{\alpha}^{k}}\Bigg\|f(g(\bm{x}^k))-\sum_{j=1}^{m_k}\alpha_j\Big(f\big(g(\bm{x}^{k-j+1})\big)- f\big(g(\bm{x}^{k-j})\big)\Big) \Bigg\|_2^2,\tag{LSP-AAf}\label{eqn:AAf_LSP}
\end{align}
respectively.  

\begin{remark}
While not considered in this paper, it is possible to change the $\ell^2$ norm used in the least-squares problems to other norms; see \cite{toth2015convergence, he2026modified} for AA-type methods and \cite{heRebholzNGMRESNSE} for NGMRES-type methods, respectively.  We comment that our unified two-grid approach can be directly extended to those cases as well, with proper adjustment in the derivations.
\end{remark}

\subsection{Two-grid Methods}
Next, we briefly recall a standard two-grid method.  We use the fixed-point iteration \eqref{eqn:fixed-point-iteration} as the smoother/relaxation scheme and assume the coarse-grid problem is defined as $A_c = RAP$, where $R$ and $P$ are the restriction and prolongation matrices, respectively.  Based on these standard components, the classical two-grid method is summarized in \cref{alg:TG_method}.

 \begin{algorithm}[H]
	\caption{Two-grid method} \label{alg:TG_method}
	\begin{algorithmic}[1]
    \STATE Given $\bm{x}^0$, $A$, $\bm{b}$, an iterative method $g(\bm{x}^k)$, a restriction matrix $R$, and an interpolation matrix $P$.
    \FOR{$k=1,2,\ldots$}
		\STATE \textbf{Presmooth $\nu_1$ iterations}: $\bm{x}^{k+\frac{1}{3}} \gets g(\bm{x}^k)$,
		\STATE $\textbf{Coarse-grid correction}$: $\bm{x}^{k+\frac{2}{3}} \gets \bm{x}^{k+\frac{1}{3}} + PA_c^{-1}R(\bm{b} - A \bm{x}^{k+\frac{1}{3}}\big)$, 
		\STATE \textbf{Postsmooth $\nu_2$ iterations}: 
 $ \bm{x}^{k+ 1}  \gets g(\bm{x}^{k+\frac{2}{3}})$.
    \ENDFOR
    \end{algorithmic}
\end{algorithm}
It is well-known that the error of the two-grid method  \cref{alg:TG_method} is given by 
\begin{equation}\label{eqn:TG-error}
    \bm{e}^{k+1} = (I-BA)^{\nu_2}(I-\pi)(I-BA)^{\nu_1}\bm{e}^{k},
\end{equation}
where $\pi = P(RAP)^{-1}RA$ is certain projection onto $\operatorname{Range}(P)$. In practice, first a proper prolongation operator $P$ is chosen, and $R$ is then determined by the type of coarse-grid projection used.  For example, when $A$ is symmetric positive definite (SPD), $\pi$ is usually chosen to be the projection with respect to the $A$ inner product, and thus $\pi = P(P^{\top}AP)^{-1} P^{\top} A$, which implies that $R = P^{\top}$.  For general $A$, if a normal equation approach is applied, a typical choice is $\pi = P(P^{\top} A^{\top}AP)^{-1}P^{\top} A^{\top}A$, i.e., the projection with respect to the $A^{\top} A$ inner product.  In this case, we have $R = P^{\top} A^{\top}$.

Now that we have established the formulation of the standard algorithms and reviewed the basic two-grid framework, we are ready to derive a two-grid formulation of all the methods. This yields a new unified algorithmic formulation and allows us to develop different variants from a point of view other than simply changing the least-squares problem. We present our derivations in \Cref{sec:AAandNGasMG}.

\section{AA and NGMRES as Two-Grid Methods}\label{sec:AAandNGasMG} 
In this section, we reformulate AA and NGMRES as two-grid methods.  The basic idea, for the linear problem, is to explicitly write out the error propagator for AA and NGMRES and then compare it with the error equation of the two-grid method \eqref{eqn:TG-error}.  This allows us to recast AA and NGMRES as two-grid methods with special choices of the coarse-grid correction.  We start by considering NGMRES and its variants.


\subsection{NGMRES and NGMRESf}
The derivation of the error equation for NGMRES applied to a linear problem was presented in \cite{he2026altGMRES}. We  include the details here for the sake of completeness.  To this end, define the prolongation matrix as
\begin{align}
	P_k^{_{NG}} &=\begin{bmatrix}
		\bm{x}^{k-1} - \bm{x}^{k} & \dots & \bm{x}^{k-m_k} - \bm{x}^{k-m_k+1} & -B\bm{r}^k 
	\end{bmatrix}.\label{eqn:NGMRES_P}
\end{align}

\begin{lemma}\label{lemma:NGMRES}
	The error of NGMRES as presented in \cref{alg:NGMRES} is given by
	\begin{equation}\label{eqn:NGMRES_error}
		\bm{e}^{k+1}   = (I- \pi_k^{_{NG}})(I-BA)\bm{e}^k,
	\end{equation}
	where  
	$$\pi_k^{_{NG}} 
	=  P_k^{_{NG}}\big( (P_k^{_{NG}})^{\top}A^{\top} A \, P_k^{_{NG}}\big)^{-1} (P_k^{_{NG}})^{\top} A^{\top}A,
	$$
	is a projection onto the coarse space $\operatorname{Range}(P_k^{_{NG}})$, with respect to the $A^\top A$ inner product.
\end{lemma}
\begin{proof}
	Note that
	\begin{align}
		r(g(\bm{x}^k)) &=A(I-BA)\bm{e}^{k},\label{eqn:rgx}\\
		r(g(\bm{x}^k))- r(\bm{x}^k) &= ABA\bm{x}^k - AB\bm{b} = -AB\bm{r}^k.\label{eqn:rgx-rx}
	\end{align}
	Using \eqref{eqn:rgx} and \eqref{eqn:rgx-rx}, we rewrite \eqref{eqn:NGMRES_LSP} as
	\begin{equation}\label{eqn:NGMRES_LSP_simplified}
		\min_{\bs{\beta}_k}\left\|A(I-BA)\bm{e}^{k} - AP_k^{_{NG}}\bs{\beta}^k\right\|_2^2.
	\end{equation}
	The solution of \eqref{eqn:NGMRES_LSP_simplified} is then
	\begin{equation}
		\bs{\beta}^k = ((P_k^{_{NG}})^{\top} A^{\top}AP_k^{_{NG}})^{-1} (P_k^{_{NG}})^{\top} A^{\top}A(I-BA)\bm{e}^{k}.\label{eqn:NGMRES_solved_beta}
	\end{equation}
	On the other hand, based on \eqref{eqn:NGMRES_update}, the error of NGMRES is given by
	\begin{align}
		\bm{e}^{k+1} &=  \bm{x} - \Big(g(\bm{x}^k) - \sum_{j = 1}^{m_k} \beta_j^{k} \big(\bm{x}^{k-j+1}- \bm{x}^{k-j} \big) - \beta_{m_k+1}^k \big( g(\bm{x}^k) - \bm{x}^k \big)\Big)\nonumber\\
		&=(I-BA)\bm{e}^k -P_k^{_{NG}}\bs{\beta}^k .\label{eqn:NGMRES_general_error}
	\end{align}
	Substituting \eqref{eqn:NGMRES_solved_beta} back into \eqref{eqn:NGMRES_general_error}, we obtain
	\begin{align}
		\bm{e}^{k+1}&= (I-BA)\bm{e}^k -P_k^{_{NG}}((P_k^{_{NG}})^{\top} A^{\top}AP_k^{_{NG}})^{-1}(P_k^{_{NG}})^{\top} A^{\top}A(I-BA)\bm{e}^{k}\nonumber\\
		&=(I- P_k^{_{NG}}((P_k^{_{NG}})^{\top} A^{\top}AP_k^{_{NG}})^{-1}(P_k^{_{NG}})^{\top} A^{\top}A)(I-BA)\bm{e}^k 
		\\ &=(I- \pi_k^{_{NG}})(I-BA)\bm{e}^k,
	\end{align}
	where $\pi_k^{_{NG}} = P_k^{_{NG}}((P_k^{_{NG}})^{\top} A^{\top}AP_k^{_{NG}})^{-1}(P_k^{_{NG}})^{\top} A^{\top}A$. This completes the proof. 
\end{proof}

Next, we present the analogous lemma for NGMRESf proposed in \cite{he2026NGMRESandAAwRich} .
\begin{lemma}\label{lemma:NGMRESf}
	The error of NGMRESf, i.e., \cref{alg:NGMRES} using
	\eqref{eqn:NGMRESf_LSP}, can be written as
	\begin{equation} \label{eqn:NGMRESf_error}
		\bm{e}^{k+1}   = (I- \pi_k^{_{NGf}})(I-BA)\bm{e}^k, 
	\end{equation}
	where \begin{align*}
		\pi_k^{_{NGf}} 
		&= P_k^{_{NG}}( (P_k^{_{NG}})^{\top} \, (A^{\top}
		B^{\top}B A)  \, P_k^{_{NG}})^{-1} (P_k^{_{NG}})^{\top} \, (A^{\top} B^{\top}BA), 
	\end{align*}
	is a projection onto the coarse space $\operatorname{Range}(P_k^{_{NG}})$ with respect to the $A^{\top} B^{\top} B A$ inner product.
\end{lemma}
\begin{proof}
	The proof is essentially the same as the proof of \cref{lemma:NGMRES}.  The main difference is that, since NGMRESf solves \eqref{eqn:NGMRESf_LSP}, which can be rewritten as
	\begin{equation*} 
		\min_{\bs{\beta}_k}\left\| BA(I-BA)\bm{e}^{k} - BA P_k^{_{NG}}\bs{\beta}^k\right\|_2^2,
	\end{equation*}
	based on \eqref{eqn:FP-residual} and \eqref{eqn:NGMRES_LSP_simplified}, \eqref{eqn:NGMRESf_error} follows directly using the same derivation as in \cref{lemma:NGMRES}.
\end{proof}

Comparing \eqref{eqn:NGMRES_error} and \eqref{eqn:NGMRESf_error} with \eqref{eqn:TG-error}, we immediately see the similarity.  In particular, if we set $\nu_1 = 1$ and $\nu_2 =0$, and choose $P = P_k^{_{NG}}$ and $R = (P_k^{_{NG}})^{\top} A^{\top} $ in the two-grid method, \cref{alg:TG_method}, we recover NGMRES.  In addition, NGMRESf is essentially a two-grid method with the choice $\nu_1 = 1$, $\nu_2 =0$, $P = P_k^{_{NG}}$, and $R = (P_k^{_{NG}})^{\top} A^{\top} B^{\top} B$.  In summary, both methods can be reformulated as two-grid methods with only one presmoothing step and a special coarse-grid correction in which the choice of the coarse space changes from iteration to iteration.  Thus, we combine the two-grid formulation of NGMRES and NGMRESf into a general algorithm presented in \cref{alg:NGMRES_as_2grid}.
\begin{algorithm}[H]
	\caption{General NGMRES($m$) as a two-grid method} \label{alg:NGMRES_as_2grid}
	\begin{algorithmic}[1]
  \FOR{$k=0,1,2,\ldots$}
        \STATE $m_k  \gets \min(k,m)$, $P_k \gets P_k^{_{NG}}$, and set
        $$
        R_k =  
        \begin{cases}
        (P_k^{_{NG}})^{\top} A^{\top}, \quad  \text{(NGMRES)} \\
        (P_k^{_{NG}})^{\top} A^{\top} B^{\top} B, \quad \text{(NGMRESf)}
        \end{cases}
        $$
		\STATE\textbf{Presmoothing}: $ \bm{x}^{k+1/2} \gets \bm{x}^{k} +B(\bm{b} - A \bm{x}^{k})$,
		\STATE \textbf{Coarse-grid correction}: 
        $$\bm{x}^{k+1} \gets \bm{x}^{k+1/2} + P_k (R_kA P_k)^{-1} R_k (\bm{b} - A \bm{x}^{k+1/2}).
        $$
\ENDFOR
	\end{algorithmic}
\end{algorithm}


\subsection{Anderson Acceleration and Variants}
As with NGMRES and NGMRESf, error equations for AA, AAr, and AAf can be derived in a similar fashion.  These error representations of AA and its variants also enable us to connect them with the two-grid method, \cref{alg:TG_method}, using different choices of smoothing steps and coarse-grid spaces.  We first define
\begin{align}
P_k^{_{AA}}&= \begin{bmatrix}
		\bm{x}^{k-1} - \bm{x}^{k} & \ldots &\bm{x}^{k-m_k} - \bm{x}^{k-m_k+1}\end{bmatrix}.\label{eqn:AA_Prolongation}  
\end{align}
We start by presenting the error lemma for the classical AA method.
\begin{lemma}\label{lemma:AA}
	The error of AA as in \cref{alg:AA} is 
	\begin{equation}\label{eqn:AA_TG_error}
		\bm{e}^{k+1} = (I -BA)(I-\pi_k^{_{AA}})\bm{e}^{k},
	\end{equation}
	where 
	\begin{align*}
		\pi_k^{_{AA}}
		&= P_k^{_{AA}}( (P_k^{_{AA}})^{\top} A^{\top}
		B^{\top}B A \, P_k^{_{AA}})^{-1} (P_k^{_{AA}})^{\top} A^{\top} B^{\top}BA, 
	\end{align*}
	is a projection onto the coarse space $\operatorname{Range}(P_k^{_{AA}})$ with respect to the $A^{\top} B^{\top} B A$ inner product.
\end{lemma}

\begin{proof}
	Since $\bm{r}^k = A\bm{e}^k$,  we rewrite \eqref{eqn:AA_LSP} as
	\begin{equation}\label{eqn:AA_LSP_simplified}
		\min_{\boldsymbol{\alpha}^{k}}\left\|BA\bm{e}^{k} - BA P_k^{_{AA}}\bs{\alpha}^k\right\|_2^2,
	\end{equation}
	whose solution is
	\begin{equation}\label{eqn:AA_solved_alpha}
		\bs{\alpha}^k = ((P_k^{_{AA}})^{\top} A^{\top}B^{\top}BAP_k^{_{AA}})^{-1}(P_k^{_{AA}})^{\top}A^{\top}B^{\top}BA\bm{e}^{k}.
	\end{equation}
	Based on \cref{alg:AA}, the AA update step \eqref{eqn:AA_update} gives 
	\begin{equation}\label{eqn:AA_general_error}
		\bm{e}^{k+1}  = (I-BA)\big(\bm{e}^k - P_k^{_{AA}}\bs{\alpha}^k\big).
	\end{equation}
	Substituting \eqref{eqn:AA_solved_alpha} into \eqref{eqn:AA_general_error} gives
	\begin{align}
		\bm{e}^{k+1} &= (I-BA)\Big(\bm{e}^k - P_k^{_{AA}}( (P_k^{_{AA}})^{\top}A^{\top}B^{\top}BAP_k^{_{AA}})^{-1} (P_k^{_{AA}})^{\top}A^{\top}B^{\top}BA\bm{e}^{k}\Big)\nonumber\\
		&=(I-BA)\big(I - \pi_k^{_{AA}}\big)\bm{e}^k.
	\end{align}
This completes the proof.
\end{proof}

Similarly, we can derive the error equation for AAr as well, which is summarized in the following lemma.
\begin{lemma}\label{lemma:AAr}
	The error of AAr, i.e., \cref{alg:AA} using \eqref{eqn:AAr_LSP}, is given by
	\begin{equation}\label{eqn:AAr_TG_error}
		\bm{e}^{k+1} = (I -BA)(I-\pi_k^{_{AAr}})\bm{e}^{k},
	\end{equation}
	where, letting $M := I-AB$,
	\begin{align*}
		\pi_k^{_{AAr}} 
		&= P_k^{_{AA}}\big((P_k^{_{AA}})^{\top} A M^{\top} M A \, P_k^{_{AA}}\big)^{-1} (P_k^{_{AA}})^{\top} A^{\top} M^{\top} MA,
	\end{align*}
	is a projection onto the coarse space $\operatorname{Range}(P_k^{_{AA}})$ with respect to the $A^{\top}M^{\top} M A$ inner product.
\end{lemma}
\begin{proof}
	Note that \eqref{eqn:AAr_LSP} can be rewritten as
	\begin{equation*}
		\min_{\boldsymbol{\alpha}^{k}}\left\|M A\bm{e}^{k} - M A P_k^{_{AA}}\bs{\alpha}^k\right\|_2^2.
	\end{equation*}
	Then, we follow the exact same procedure as in the proof of \cref{lemma:AA} to obtain \eqref{eqn:AAr_TG_error}.
\end{proof}

Finally, we consider AAf and summarize its error equation in the following lemma
\begin{lemma}\label{lemma:AAf}
	The error of AAf, i.e., \cref{alg:AA} using \eqref{eqn:AAf_LSP}, is given by
	\begin{equation}\label{eqn:AAf_TG_error}
		\bm{e}^{k+1} = (I -BA)(I-\pi_k^{_{AAf}})\bm{e}^{k},
	\end{equation}
	where, letting $M := B(I-AB)$,
	\begin{align*}
		\pi_k^{_{AAf}} 
		&= P_k^{_{AA}}\big((P_k^{_{AA}})^{\top} A M^{\top} M A \, P_k^{_{AA}}\big)^{-1} (P_k^{_{AA}})^{\top} A^{\top} M^{\top} MA,
	\end{align*}
	is a projection onto the coarse space $\operatorname{Range}(P_k^{_{AA}})$ with respect to the $A^{\top}M^{\top} M A$ inner product.
\end{lemma}
\begin{proof}
	The proof is the same as before by noting that \eqref{eqn:AAf_LSP} can be rewritten as
	$	\min_{\boldsymbol{\alpha}^{k}}\left\|M A\bm{e}^{k} - M AP_k^{_{AA}}\bs{\alpha}^k\right\|_2^2$ with $M := B(I-AB)$.
\end{proof}

Again, the above results reveal that AA, AAr, and AAf can all be expressed as two-grid methods with a coarse-grid correction followed by a postsmoothing step.  They differ only in the coarse-grid correction step, i.e., they use different projections onto $\operatorname{Range}(P)$, where $P = P_{k}^{_{AA}}$, with respect to different inner products.  This leads to different choices of the restriction $R$.  In particular, the restrictions are $(P_k^{AA})^{\top} A^{\top} B^{\top} B$, $(P_k^{AA})^{\top} A^{\top} (I-AB)^{\top} (I-AB)$, and $(P_k^{AA})^{\top} A^{\top} (I-AB)^{\top} B^{\top} B (I-AB)$ for AA, AAr, and AAf, respectively.  We present the following general AA algorithm, covering AA, AAr, and AAf, in \cref{alg:AA_as_2grid}.
\begin{algorithm}[H]
	\caption{General AA($m$) as a two-grid method (at $k$-th iteration)} \label{alg:AA_as_2grid}
	\begin{algorithmic}[1]
  \FOR{k=1,2,\ldots}
  \STATE $m_k \gets \min(k,m)$,  $P_k \gets P_k^{_{AA}}$, and 
  $$
  R_k \gets 
  \begin{cases}
  	(P_k^{AA})^{\top} A^{\top}B^{\top}B,  \quad  \text{(AA)} \\
  	(P_k^{AA})^{\top} A^{\top} (I-AB)^{\top}(I-AB), \quad \text{(AAr)} \\
  	(P_k^{AA})^{\top} A^{\top}(I-AB)^{\top}B^{\top}B(I-AB), \quad \text{(AAf)} \\
  \end{cases}.
  $$
\STATE $\textbf{Coarse-grid\ correction}$: $\bm{x}^{k+1/2} \gets \bm{x}^k + P_k(R_kAP_k)^{-1} R_k (\bm{b} - A \bm{x}^k)$,
\STATE $\textbf{Postsmoothing}$: $ \bm{x}^{k+1} \gets \bm{x}^{k+1/2} + B(\bm{b} - A \bm{x}^{k+1/2})$.
\ENDFOR
\end{algorithmic}
\end{algorithm}

\section{A General Two-Grid Acceleration Method}\label{sec:TGA}
AA, NGMRES, and all the presented variants can be expressed in terms of a smoothing step, either pre- or postsmoothing, and a coarse-grid correction.  We summarize the results for all the variants in \cref{tab:Choice_of_P_and_R}.

\begin{table}[h!]
    \centering
    \footnotesize
    \begin{tabular}{|c|c|c|c|}
\midrule
\textbf{Method} & \multicolumn{3}{c|}{$P_k$}  \\
\midrule
\midrule
NGMRES \& NGMRESf & \multicolumn{3}{c|}{$\begin{bmatrix}
        \bm{x}^{k-1} - \bm{x}^{k} & \dots & \bm{x}^{k-m_k} - \bm{x}^{k-m_k+1} & -B\bm{r}^k
    \end{bmatrix}$} \\
\midrule
AA, AAr, \& AAf & \multicolumn{3}{c|}{$\begin{bmatrix}
        \bm{x}^{k-1} - \bm{x}^{k} & \dots & \bm{x}^{k-m_k} - \bm{x}^{k-m_k+1}
    \end{bmatrix}$}\\
\midrule
\midrule
& Error propagation & \multicolumn{2}{c|}{$\pi_k$ is a projection onto $\operatorname{Range}(P_k)$ w.r.t. }\\
\midrule
\midrule
NGMRES &$(I-\pi_k)(I-BA)$&
 \multicolumn{2}{c|}{$A^{\top}A$ inner product} \\
 \midrule
NGMRESf &$(I-\pi_k)(I-BA)$&
 \multicolumn{2}{c|}{$A^{\top}B^{\top}BA$ inner product} \\
\midrule
AA & $(I-BA)(I-\pi_k)$& \multicolumn{2}{c|}{$A^{\top}B^{\top}BA$ inner product} \\
\midrule
AAr& $(I-BA)(I-\pi_k)$& \multicolumn{2}{c|}{$ A^{\top}(I-AB)^{\top}(I-AB)A$ inner product } \\
\midrule
AAf& $(I-BA)(I-\pi_k)$& \multicolumn{2}{c|}{$A^{\top}(I-AB)^{\top}B^{\top}B(I-AB)A$ inner product} \\
\midrule
    \end{tabular}
    \caption{Summary of choices of two-grid components for general AA and NGMRES as two-grid algorithms.}    
    \label{tab:Choice_of_P_and_R}
\end{table}

The above results and observations from \Cref{tab:Choice_of_P_and_R}
 motivate a natural question: \emph{what other choices of smoothing and coarse-grid correction could be used to accelerate the fixed-point iteration \eqref{eqn:fixed-point-iteration}?} Addressing this question, we propose a unified two-grid acceleration framework, which not only recovers AA, NGMRES, and their existing variants as special cases, but also enables us to derive new variants from a different point of view, such as by adjusting the smoothing steps or by modifying the coarse space defined by $\operatorname{Range}(P)$.
This new formulation is presented in \cref{alg:TGA(m)}.
\begin{algorithm}[H]
	\caption{Two-grid Accelerator: TGA${(m;\nu_1,\nu_2)}$} 
	\label{alg:TGA(m)}
	\begin{algorithmic}[1]
		\STATE Given $\bm{x}^0$
		\FOR{$k=1,2,\ldots$}
		\STATE \textbf{Presmooth $\nu_1$ times}: $\bm{x}^{k+\frac{1}{3}} \gets \bm{x}^{k} + B\big(\bm{b} - A \bm{x}^{k}\big)$,
		\STATE \textbf{Compute the residual:} $\bm{r}^{k+\frac{1}{3}} \gets \bm{b} - A \bm{x}^{k + \frac{1}{3}}$,
		\STATE $m_k \gets \min(m,k)$; choose $P_k$ and then construct $R_k$ based on the choice of projection onto the coarse space $\operatorname{Range}(P_k)$,
		\STATE \textbf{Coarse-grid correction}: $\bm{x}^{k+\frac{2}{3}} \gets \bm{x}^{k+\frac{1}{3}} + P_k (R_k A P_k)^{-1} R_k \bm{r}^{k+\frac{1}{3}}$, 
		\STATE \textbf{Postsmooth $\nu_2$ times}: 
		$ \bm{x}^{k+ 1}  \gets \bm{x}^{k+\frac{2}{3}} + B\big(\bm{b} - A \bm{x}^{k+\frac{2}{3}}\big)$.
		\ENDFOR
	\end{algorithmic}
\end{algorithm}

Again, we see that different choices of $P_k$ lead to NGMRES and AA, while the choice of smoothing (pre- versus postsmoothing) distinguishes the two methods as well.  Their respective variants, although originally proposed from the point of view of different least-squares problems, can equivalently be understood within our framework as arising from different choices of coarse-grid correction, namely, projections onto the same coarse space $\operatorname{Range}(P_k)$ but with respect to different inner products.  This reinterpretation is practically useful.  Rather than requiring a new least-squares problem to be formulated for each variant, one only needs to specify the inner product defining the projection $\pi_k$, and the corresponding restriction $R_k$ follows automatically.  This observation is precisely what allows us to unify AA, NGMRES, and all of their variants within the single algorithmic framework of \cref{alg:TGA(m)}.

The main advantage of our two-grid framework is that it provides the flexibility to propose new variants based on multigrid methodology, rather than being restricted to reformulating the underlying least-squares problem.  In particular, we can easily develop new variants by independently adjusting different components of the two-grid algorithm, such as the choice of coarse space, the projection (and hence the inner product), as well as the number and placement of smoothing steps.  In this work, we mainly focus on one such variant, with a specific choice of the prolongation matrix $P_k$ and pre- and postsmoothing steps.  We comment that there are other possible adjustments one can explore, such as multilevel and nonlinear extensions, which are the subject of ongoing work.

\subsection{A Specific Variant}\label{sec:TGAvariant}
To construct our special variant, we first discuss the choice of the prolongation matrix. This is an important component of the two-grid algorithm, since the effectiveness of a two-grid method relies on how well the coarse space captures the low-frequency (or difficult-to-converge) error components.  Comparing the prolongation matrices $P_{k}^{AA}$ and $P_k^{NG}$ given in \cref{tab:Choice_of_P_and_R}, one can see that $P_k^{NG}$ has one extra column, $-B\bm{r}^k$, which partially explains the good convergence behavior of NGMRES.  Under our two-grid framework, this extra column may capture more information about the low-frequency error components that are difficult to converge, since $B\bm{r}^k \approx A^{-1}\bm{r}^k = \bm{e}^k$.  Thus, for our two-grid accelerator, we construct $P_k$ in a similar fashion, including such an extra column.  However, if a presmoothing step is used, the most recently updated residual is $\bm{r}^{k+\frac{1}{3}}$, not $\bm{r}^k$.  We emphasize that this is hidden in the original NGMRES algorithm \cref{alg:NGMRES} and is only revealed by our two-grid formulation \cref{alg:TGA(m)}.  We therefore propose to include $B\bm{r}^{k+\frac{1}{3}}$ in the construction.  To keep the dimension of the coarse space the same, we also drop the oldest iterate and construct the following prolongation:
\begin{align} 
	& \quad P_k = P_k^{_{TGA}}  \nonumber \\
	& \label{def:P_TGA} := 
	\begin{cases}
		[-B \bm{r}^k, -B \bm{r}^{k+\frac{1}{3}}], & \text{if } \nu_1 > 0, \ m =1, \\
		[\bm{x}^{k-1} - \bm{x}^k, \cdots, \bm{x}^{k - m_k+1}- \bm{x}^{k - m_k+2}, -B \bm{r}^k, -B\bm{r}^{k+\frac{1}{3}}], & \text{if } \nu_1 > 0,  \ m \geq 2, \\
		[\bm{x}^{k-1} - \bm{x}^k, \cdots, \bm{x}^{k - m_k}- \bm{x}^{k - m_k+1}, -B \bm{r}^k], & \text{if } \nu_1 = 0.
	\end{cases}
\end{align}

Besides the choice of $P_k$, we can also vary the smoothing steps.  That is, while the original AA and NGMRES algorithms perform only one step of post- or presmoothing, respectively, our new two-grid formulation allows us to adjust the number and placement of smoothing steps as needed.  This provides a class of variants of NGMRES and AA.  However, in order to make our two-grid accelerator computationally comparable with NGMRES, AA, and their variants, we mainly consider the simple setting $\nu_1 = 1$ and $\nu_2 = 0$ in our numerical experiments.  Note that this choice matches NGMRES and allows us to use $P_{k}^{_{TGA}}$, as defined in \eqref{def:P_TGA}.

To summarize, although we can develop many different variants based on our two-grid framework, we mainly consider variants that use the prolongation $P_k^{_{TGA}}$.  For simplicity, the projection $\pi_k$ is defined with respect to the $A^{\top}A$ inner product, which gives $R_k = (P_k^{_{TGA}})^{\top} A^{\top}$.  We use TGA$(m;\nu_1,\nu_2)$ to denote our general two-grid accelerator, and mainly focus on the case $\nu_1 = 1$ and $\nu_2=0$ for a fair comparison.


\section{Numerical Results}\label{sec:numRes} 
In this section, we present numerical results to verify the effectiveness of TGA$(m;\nu_1,\nu_2)$ as an accelerating framework.  We first consider the 2D Poisson problem to verify that \Cref{alg:NGMRES_as_2grid} and \Cref{alg:AA_as_2grid} are equivalent to NGMRES and AA, respectively.  We then test the two-grid accelerator discussed in \Cref{sec:TGA} on this 2D Poisson problem, as well as on a convection-diffusion equation.  All methods were coded in MATLAB R2025b, and a direct solver was used whenever matrix inversion was needed.


\subsection{SPD Example: Poisson's Equation} \label{sec:SPD-example}
We consider the 2D Poisson equation, $-\Delta \bm{u} = f$, defined on the unit square $[0,1] \times [0,1]$ with zero Dirichlet boundary conditions.  The problem is discretized using linear finite elements on a uniform triangular mesh with $N+2$ points in each direction, giving a grid spacing of $h = 1/(N+1)$.  After eliminating boundary conditions, this yields a matrix that is $n\times n$, where $n=N^2$.  In our experiments, we consider a mesh of size $16 \times 16$ ($h=1/17$), corresponding to a matrix of size $256\times 256$, and a mesh of size $32 \times 32$ ($h=1/33$), corresponding to a matrix of size $1024\times 1024$.  We fix the right-hand side to be a vector of ones, $\bm{b} = \bm{1}$, and assume $B= \omega I$ with $\omega = 1/||A||_1 = 1/8$ so that $\rho(I-BA) <1$. The initial guess $\bm{x}^0$ is randomly generated via MATLAB's \texttt{rand} function, and the stopping tolerance is a relative reduction in residual of $10^{-10}$, i.e. $\|\bm{r}^k\|/\|\bm{r}^0\| \leq 10^{-10}$.

We first demonstrate the equivalence of the TGA formulations in \cref{alg:AA_as_2grid} (referred to as TGA$(m;0,1)$-AA) and \cref{alg:NGMRES_as_2grid} (referred to as TGA$(m;1,0)$-NG) with the standard AA and NGMRES formulations, respectively.  Here we use the $P_k$ listed in \Cref{tab:Choice_of_P_and_R} so that the correspondingTGA formulations corresponds to the standard AA and NGMRES.

\begin{figure}[H]
\centering
\includegraphics[width = .7\textwidth]{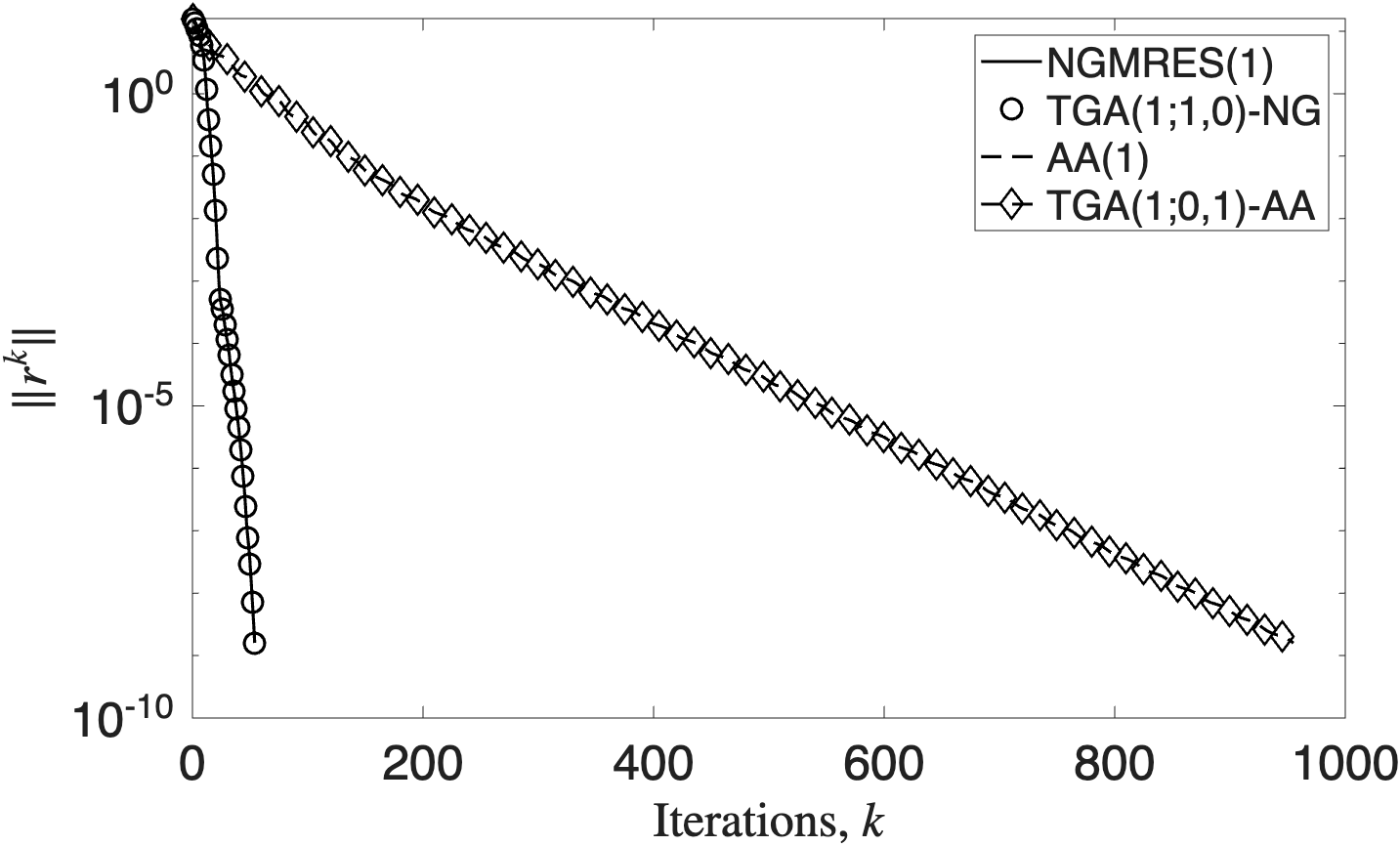}
\caption{\small Plots of $\|\bm{r}^k\|_2$ versus iterations to verify the two-grid reformulations of AA(1) and NGMRES(1) on the 2D Poisson with matrix size $256 \times 256$.}
\label{fig:TGA-equivalence}
\end{figure}

From \cref{fig:TGA-equivalence}, we observe the expected equivalences: TGA$(1;1,0)$-NG is identical to standard NGMRES$(1)$, and TGA$(1;0,1)$-AA and AA$(1)$ are equivalent as well.  We also observe that NGMRES$(1)$ converges faster than AA$(1)$, which is expected.  This faster convergence occurs because NGMRES$(1)$ produces iterates identical to GMRES when the matrix $A$ is SPD, as shown in \cite{greif2026convergenceLinGMRES}.

We proceed to explore the specific variant TGA$(m;1,0)$, which uses the prolongation \eqref{def:P_TGA} developed in \Cref{sec:TGAvariant}, and compare it with NGMRES, AA, and their variants.  We use a Jacobi smoother throughout, i.e., $B = D^{-1}$ with $D = \operatorname{diag}(A)$.  We test using two different matrix sizes, $n=256$ and $n=1024$, and report the number of iterations for $m=1,2,5,10$ in \cref{tab:exp2_avg_its}.

\begin{table}[H]
	\centering
	\begin{tabular}{cc}
	\begin{tabular}{|c|c|c|}
		\hline \hline 
		& $n =256$ & $n =1024$\\
		\hline
		AA(1)& 1240  & 4696\\
		\hline
		AAr(1)& 1058 & 4215 \\
		\hline
		AAf(1)& 1058 & 4215 \\
		\hline
	    NGMRES(1)& 56 & 108 \\
		\hline
		NGMRESf(1) & 56 & 108 \\
		\hline
		TGA(1;1,0)& 314 & 1063 \\
		\hline\hline
			\end{tabular}
			\begin{tabular}{|c|c|c|}
				\hline \hline
		& $n =256$ & $n =1024$\\
		\hline
		AA(2)& 729 & 2843 \\
		\hline
		AAr(2)& 638  & 2433\\
		\hline
		AAf(2)& 638 & 2433 \\
		\hline
		NGMRES(2)& 56 & 108 \\
		\hline
		NGMRESf(2) & 56 & 108 \\
		\hline
		TGA(2;1,0)& 41  & 175 \\
		\hline\hline
	\end{tabular}
	\\[1em]
	\begin{tabular}{|c|c|c|}
		\hline \hline 
		& $n =256$ & $n =1024$\\
		\hline
		AA(5)& 267  & 1239\\
		\hline
		AAr(5)& 258 & 1251 \\
		\hline
		AAf(5)& 258 & 1251 \\
		\hline
		NGMRES(5)& 56 & 108 \\
		\hline
		NGMRESf(5) & 56 & 108 \\
		\hline
		TGA(5;1,0)& 41 & 79 \\
		\hline\hline
	\end{tabular}
	\begin{tabular}{|c|c|c|}
		\hline \hline
		& $n =256$ & $n =1024$\\
		\hline
		AA(10)& 148  & 416\\
		\hline
		AAr(10)& 147  & 394\\
		\hline
		AAf(10)& 147 & 394 \\
		\hline
		NGMRES(10)& 56 & 108 \\
		\hline
		NGMRESf(10) & 56 & 108 \\
		\hline
		TGA(10;1,0)& 36 & 64 \\
		\hline\hline
	\end{tabular}
\end{tabular}
	\caption{ \small Comparison of iterations needed to achieve $\|\bm{r}^k\|_2/\|\bm{r}^0\|_2 < 10^{-10}$ for TGA$(m; 1,0)$ against AA, NGMRES, and their variants on 2D Poisson problem.}
	\label{tab:exp2_avg_its}
\end{table}

From \Cref{tab:exp2_avg_its}, we see that AAr$(m)$ and AAf$(m)$ require the same number of iterations for each value of $m$.  This is because, in our setting, $B = D^{-1} = \frac{1}{4}I$, which is simply a scaling of the identity matrix; therefore, the AA variants are equivalent.  For the same reason, NGMRES$(m)$ and NGMRESf$(m)$ also require the same number of iterations for each value of $m$.  However, since as discussed above NGMRES$(m)$ is equivalent to GMRES for the SPD problem for any $m \geq 1$, the number of iterations stays the same for all window sizes; the same holds for NGMRESf$(m)$, since it is equivalent to NGMRES$(m)$ in our setting.  In contrast, the iteration counts for AA, AAr, and AAf decrease as $m$ increases, though they remain far larger than those of NGMRES for all tested values of $m$. 

On the other hand, our new variant, TGA$(m;1,0)$, outperforms all the other methods when $m\geq 2$ and also benefits from increasing $m$.  It even achieves better performance than NGMRES$(m)$ (and NGMRESf$(m)$) in this example, because the two-grid formulation \cref{alg:TGA(m)} allows us to bring in the most recently updated information through $B\bm{r}^{k+\frac{1}{3}}$, which provides a richer coarse space without increasing the dimension for projecting out the low-frequency error components in the multilevel framework.  We also note that, when $m=1$, TGA$(1;1,0)$ requires more iterations to converge than NGMRES$(1)$ (and NGMRESf$(1)$).  This is expected, since the equivalence between NGMRES$(1)$ and GMRES requires a coarse space spanned by at least $\bm{x}^{k-1} - \bm{x}^k$ and $B\bm{r}^k$, whereas TGA$(1;1,0)$ uses a coarse space spanned by $\bm{x}^{k-1}-\bm{x}^k$ and $B\bm{r}^{k+\frac{1}{3}}$.  To summarize, our new variant TGA$(m;1,0)$ is effective when $m \geq 2$ for this SPD example, compared to NGMRES, AA, and their variants. 

\subsection{Non-SPD Example: the Convection-Diffusion Equation}
Next we consider the two dimensional convection-diffusion problem defined as 
\begin{equation*}
   -\eta \Delta u  + \bm{b} \cdot \nabla u = f, \quad  \text{in }  \Omega = [0,1] \times [0,1],
    \label{eqn:convection_diff}
\end{equation*}
with $\bm{b} = (0,1)^{\top}$  and homogeneous Dirichlet boundary conditions. Here $\eta$ denotes the diffusion coefficient and determines whether the problem is convection or diffusion dominated. A small value of $\eta \ll 1$ corresponds to a convection-dominated problem.  We use the same uniform triangular mesh as the SPD example in \Cref{sec:SPD-example} and employ the usual upwind scheme.  Similarly,  in our experiments, we consider a mesh of size $h=1/17$, corresponding to a matrix of size $256\times 256$, and a mesh of size $h=1/33$, corresponding to a matrix of size $1024\times 1024$.  We still use $\bm{b} = \bm{1}$ as the right-hand side, a random initial guess $\bm{x}^0$, and the stopping tolerance is still $\|\bm{r}^k\|/\|\bm{r}^0\| \leq 10^{-10}$.  We again use a Jacobi smoother here,  i.e., $B = D^{-1}$ with $D = \operatorname{diag}(A)$, and report the number of iterations for $\eta = 1$, $10^{-3}$, and $10^{-6}$ in \cref{tab:exp_non_SPD_eta_1}, \cref{tab:exp_non_SPD_eta_1e-3}, and \cref{tab:exp_non_SPD_eta_1e-6}, respectively.

\begin{table}[H]
	\centering
	\begin{tabular}{cc}
		\begin{tabular}{|c|c|c|}
			\hline \hline 
			& $n =256$ & $n =1024$\\
			\hline
			AA(1)& 1096  & 3871 \\
			\hline
			AAr(1)& 1099 & 3867 \\
			\hline
			AAf(1)& 1099 & 3867 \\
			\hline
			NGMRES(1)& 125 & 263 \\
			\hline
			NGMRESf(1) & 125 & 263 \\
			\hline
			TGA(1;1,0)& 279 &  966 \\
			\hline\hline
		\end{tabular}
		\begin{tabular}{|c|c|c|}
			\hline \hline
			& $n =256$ & $n =1024$\\
			\hline
			AA(2)& 635 & 2843 \\
			\hline
			AAr(2)& 575  & 2433\\
			\hline
			AAf(2)& 575 & 2433 \\
			\hline
			NGMRES(2)& 116 & 108 \\
			\hline
			NGMRESf(2) & 116 & 108 \\
			\hline
			TGA(2;1,0)& 59  & 120 \\
			\hline\hline
		\end{tabular}
		\\[1em]
		\begin{tabular}{|c|c|c|}
			\hline \hline 
			& $n =256$ & $n =1024$\\
			\hline
			AA(5)& 281  & 1210\\
			\hline
			AAr(5)& 278 & 1202 \\
			\hline
			AAf(5)& 278 & 1202\\
			\hline
			NGMRES(5)& 97 & 208 \\
			\hline
			NGMRESf(5) & 97 & 208 \\
			\hline
			TGA(5;1,0)& 44 & 98 \\
			\hline\hline
		\end{tabular}
		\begin{tabular}{|c|c|c|}
			\hline \hline
			& $n =256$ & $n =1024$\\
			\hline
			AA(10)& 128  & 409 \\
			\hline
			AAr(10)& 121  & 426 \\
			\hline
			AAf(10)& 121 & 426 \\
			\hline
			NGMRES(10)& 84 & 181 \\
			\hline
			NGMRESf(10) & 84 & 181 \\
			\hline
			TGA(10;1,0)& 41 & 78 \\
			\hline\hline
		\end{tabular}
	\end{tabular}
	\caption{ \small $\eta = 1$}
	\label{tab:exp_non_SPD_eta_1}
\end{table}

When $\eta = 1$, we are in the diffusion-dominated regime.  Thus, from \cref{tab:exp_non_SPD_eta_1}, we observe behavior similar to the SPD example.  First, AAr and AAf are equivalent, and NGMRES and NGMRESf are equivalent.  This is again because $B$ is a scaling of the identity matrix for this example as well.  Second, NGMRES and NGMRESf are, in general, more effective than AA and its variants.  Meanwhile, our new variant TGA outperforms all the other methods when $m \geq 2$, demonstrating the benefit of including the most recently updated residual $\bm{r}^{k+\frac{1}{3}}$ in the coarse space.  Finally, we see that TGA$(1;1,0)$ still requires more iterations than NGMRES and NGMRESf in this diffusion-dominated regime.

\begin{table}[H]
	\centering
	\begin{tabular}{cc}
		\begin{tabular}{|c|c|c|}
			\hline \hline 
			& $n =256$ & $n =1024$\\
			\hline
			AA(1)&  37  & 74\\
			\hline
			AAr(1)& 38 & 75 \\
			\hline
			AAf(1)& 38 & 75 \\
			\hline
			NGMRES(1)& 56 & 95 \\
			\hline
			NGMRESf(1) & 56 & 95 \\
			\hline
			TGA(1;1,0)& 29 & 47 \\
			\hline\hline
		\end{tabular}
		\begin{tabular}{|c|c|c|}
			\hline \hline
			& $n =256$ & $n =1024$\\
			\hline
			AA(2)& 41 & 80 \\
			\hline
			AAr(2)& 41  & 80\\
			\hline
			AAf(2)& 41 & 80 \\
			\hline
			NGMRES(2)& 61 & 104 \\
			\hline
			NGMRESf(2) & 61 & 104 \\
			\hline
			TGA(2;1,0)& 40  & 52 \\
			\hline\hline
		\end{tabular}
		\\[1em]
		\begin{tabular}{|c|c|c|}
			\hline \hline 
			& $n =256$ & $n =1024$\\
			\hline
			AA(5)& 55  & 99\\
			\hline
			AAr(5)& 54 & 100 \\
			\hline
			AAf(5)& 54 & 100 \\
			\hline
			NGMRES(5)& 75 & 132 \\
			\hline
			NGMRESf(5) & 75 & 132 \\
			\hline
			TGA(5;1,0)& 43 & 85 \\
			\hline\hline
		\end{tabular}
		\begin{tabular}{|c|c|c|}
			\hline \hline
			& $n =256$ & $n =1024$\\
			\hline
			AA(10)& 53  & 120\\
			\hline
			AAr(10)& 48  & 114\\
			\hline
			AAf(10)& 48 & 114 \\
			\hline
			NGMRES(10)& 83 & 160 \\
			\hline
			NGMRESf(10) & 83 & 160 \\
			\hline
			TGA(10;1,0)& 36 & 86 \\
			\hline\hline
		\end{tabular}
	\end{tabular}
	\caption{ \small $\eta = 10^{-3}$}
	\label{tab:exp_non_SPD_eta_1e-3}
\end{table}

For $\eta = 10^{-3}$, the problem is moderately convection-dominated, and the results in \cref{tab:exp_non_SPD_eta_1e-3} look noticeably different from the diffusion-dominated case.  First, unlike the SPD and diffusion-dominated cases, AA and its variants now converge faster than NGMRES and NGMRESf across all window sizes $m$.  Second, TGA$(m;1,0)$ remains superior in all cases. It achieves the best iteration count among all methods at every $m$, and even at $m=1$ it already outperforms AA, NGMRES, and their variants.  This indicates that, in this more convection-dominated regime, including the most recently updated residual $\bm{r}^{k+\frac{1}{3}}$ is even more effective than in the diffusion-dominated case.  Finally, for all the acceleration methods, the iteration count grows with $m$, indicating that, in this more convection-dominated regime, increasing the window size does not uniformly improve the accelerators' performance.

\begin{table}[H]
	\centering
	\begin{tabular}{cc}
		\begin{tabular}{|c|c|c|}
			\hline \hline 
			& $n =256$ & $n =1024$\\
			\hline
			AA(1)& 30  & 57\\
			\hline
			AAr(1)& 29 & 55 \\
			\hline
			AAf(1)& 29 & 55 \\
			\hline
			NGMRES(1)& 60 & 101 \\
			\hline
			NGMRESf(1) & 60 & 101 \\
			\hline
			TGA(1;1,0)& 28 & 49 \\
			\hline\hline
		\end{tabular}
		\begin{tabular}{|c|c|c|}
			\hline \hline
			& $n =256$ & $n =1024$\\
			\hline
			AA(2)& 38 & 74 \\
			\hline
			AAr(2)& 38  & 72\\
			\hline
			AAf(2)& 38 & 72 \\
			\hline
			NGMRES(2)& 77 & 112 \\
			\hline
			NGMRESf(2) & 77 & 112 \\
			\hline
			TGA(2;1,0)& 35 & 58 \\
			\hline\hline
		\end{tabular}
		\\[1em]
		\begin{tabular}{|c|c|c|}
			\hline \hline 
			& $n =256$ & $n =1024$\\
			\hline
			AA(5)& 59  & 94\\
			\hline
			AAr(5)& 54 & 96 \\
			\hline
			AAf(5)& 54 & 96 \\
			\hline
			NGMRES(5)& 88 & 138 \\
			\hline
			NGMRESf(5) & 88 & 138 \\
			\hline
			TGA(5;1,0)& 50 & 77 \\
			\hline\hline
		\end{tabular}
		\begin{tabular}{|c|c|c|}
			\hline \hline
			& $n =256$ & $n =1024$\\
			\hline
			AA(10)& 54  & 125 \\
			\hline
			AAr(10)& 42  & 118 \\
			\hline
			AAf(10)& 42 & 118 \\
			\hline
			NGMRES(10)& 84 & 165 \\
			\hline
			NGMRESf(10) & 84 & 165 \\
			\hline
			TGA(10;1,0)& 36 & 94 \\
			\hline\hline
		\end{tabular}
	\end{tabular}
	\caption{ \small $\eta = 10^{-6}$}
	\label{tab:exp_non_SPD_eta_1e-6}
\end{table}

For $\eta = 10^{-6}$, the problem is strongly convection-dominated, and the results in \cref{tab:exp_non_SPD_eta_1e-6} exhibit the same qualitative trends observed for $\eta = 10^{-3}$ in \cref{tab:exp_non_SPD_eta_1e-3}, but more pronounced.  NGMRES and NGMRESf now perform even worse relative to AA and its variants. TGA$(m;1,0)$ again achieves the best performance among all methods at every $m$ and every problem size tested, and its advantage over NGMRES widens as $\eta$ decreases, underscoring the benefit of incorporating the most recently updated residual $\bm{r}^{k+\frac{1}{3}}$ as the problem becomes increasingly convection-dominated.  As in the $\eta=10^{-3}$ case, however, the iteration counts for all methods, including TGA$(m;1,0)$, still grow with $m$ rather than decrease, confirming that a larger window size is not universally beneficial for convection-dominated problems.

\section{Conclusion}\label{sec:conclusion}
In this paper, we established a new connection between Anderson acceleration, NGMRES, and classical two-grid methods.  By explicitly deriving the error propagation matrices for AA and NGMRES on linear problems, we showed that both methods, along with their existing variants (AAr, AAf, and NGMRESf), can be reformulated as two-grid methods with a single presmoothing or postsmoothing step and a coarse-grid correction defined by a projection with respect to a suitable inner product.  This unified two-grid perspective not only recovers these existing methods as special cases but also reveals that their differences can be understood purely in terms of the choice of coarse space and the inner product defining the coarse-grid projection, rather than through separate least-squares formulations.

Building on this reformulation, we proposed a general two-grid accelerator that provides a flexible framework for constructing new acceleration methods.  In particular, we introduce a new variant that augments the coarse space with the most recently updated residual information, made explicit through the two-grid reformulation, but hidden in the standard formulation of the NGMRES and AA algorithms.  Numerical experiments on both an SPD Poisson problem and a family of non-SPD convection-diffusion problems demonstrated that this new variant often outperforms both AA and NGMRES and their variants, with its advantage growing as the problem becomes increasingly convection-dominated.

The two-grid perspective developed here opens several directions for future work.  In particular, the multilevel and nonlinear extensions of the TGA framework are natural next steps and are the subject of ongoing work.  We also expect that the flexibility of the two-grid framework to accommodate different smoothing strategies and coarse-space constructions will enable further new variants of AA and NGMRES tailored to specific problem classes.

\bibliographystyle{siamplain}
\bibliography{references}

\end{document}